\documentclass[a4paper]{article}
\usepackage{fullpage}
\usepackage[T1]{fontenc}
\usepackage{amsmath,amssymb,amsthm}
\usepackage{cite}
\usepackage{graphicx}
\usepackage{authblk}
\usepackage{subfig}
\usepackage{algorithm}
\usepackage[noend]{algpseudocode}
\usepackage[breaklinks=true]{hyperref}

\newcommand\Tan{\qopname\relax o{Tan}}
\newtheorem{theorem}{Theorem}

\title{Longest Straight Line Paths on Water or Land on the Earth}

\author[1]{Rohan Chabukswar\thanks{ChabukR@utrc.utc.com}}
\author[2]{Kushal Mukherjee\thanks{KushMukh@in.ibm.com}}
\affil[1]{United Technologies Research Center Ireland}
\affil[2]{IBM Research India}

\begin{document}

\maketitle

\begin{abstract}
  There has been some interest recently in determining the longest distance one can sail for on the earth without hitting land, as well as in the converse problem of determining the longest distance one could drive for on the earth without encountering a major body of water. In its basic form, this is an optimisation problem, rendered chaotic by the presence of islands and lakes, and indeed the fractal nature of the coasts. In this paper we present a methodology for calculating the two paths using the branch-and-bound algorithm.
\end{abstract}

\section{Introduction}
On December 29, 2012, Reddit user \texttt{kepleronlyknows} posted a map (Figure \ref{fig:rQlk4}) to r/MapPorn \cite{kepleronlyknows} showing the longest straight line that can be sailed for on the earth without hitting land, from Pakistan to eastern Russia. This was generated a lot of interest and led to subsequent attempts to prove and disprove the user, along with discussions on solving the converse problem of determining the longest distance that can be driven on land without hitting a major water body.

\begin{figure}[htbp]
  \centering
  \includegraphics[width=\textwidth]{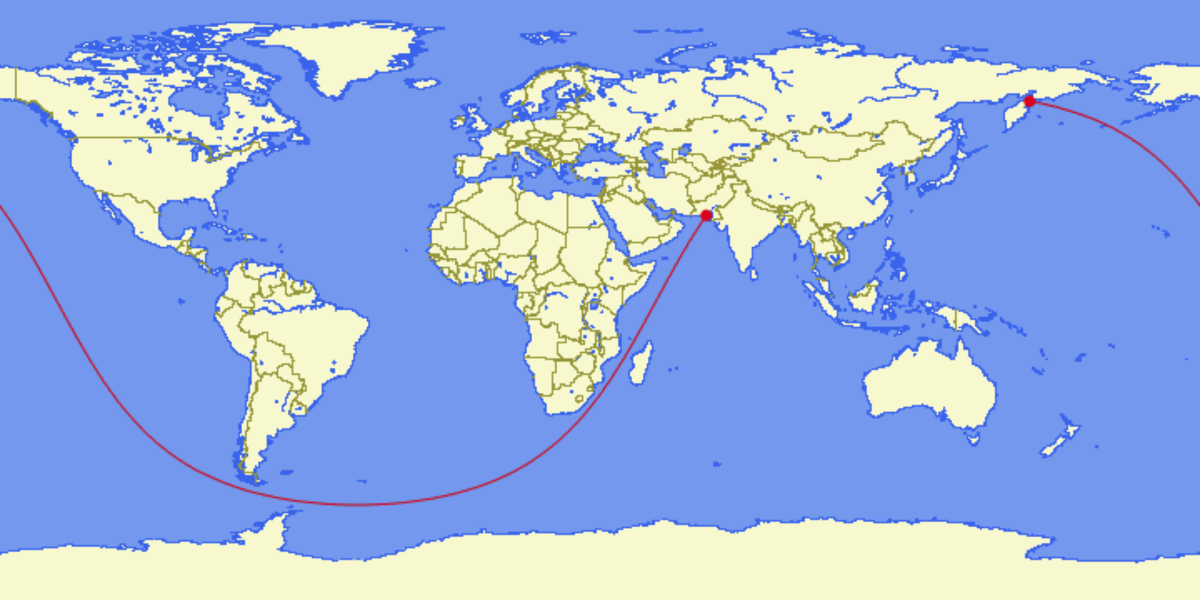}
  \caption{\texttt{kepleronlyknows}'s original map posted to r/MapPorn}
  \label{fig:rQlk4}
\end{figure}

In their basic forms, the two are optimisation problems in two dimensions --- however, the objective functions are discontinuous due to the presence of islands (in the case of the longest sailable path), and lakes (in the case of the longest drivable path). Globally optimal solutions would require an exhaustive search over the two dimensions, a not inconsiderable obstacle in terms of time and memory. For a 1-arcminute resolution ($\sim$1.855 km at the equator), which is the resolution provided by the ETOPO1 Global Relief model of Earth's surface by National Oceanic and Atmospheric Administration, there would be 233,280,000 great circles to consider to find the global optimum, and each great circle would have 21600 individual points to process --- a staggering 5,038,848,000,000 points to verify.
\subsection{Previous Work}
Apart from the discovery of the longest sailable path by \texttt{kepleronlyknows}, Reddit user \texttt{Groke} also posted \cite{Groke} a path from Norway to Antarctica in a straight line without hitting land, which is shorter than the path from Pakistan to Eastern Russia. YouTube user David Cooke claimed \cite{Cooke} to have discovered an even longer path (`The Cooke Passage') from Quebec to British Columbia, which was later proved \cite{GeoGarage} to not be a straight line in a blog post at GeoGarage, a nautical charts platform.

Guy Bruneau of IT/GIS Consulting services calculated \cite{Bruneau} a path from Eastern China to Western Liberia as being the longest distance you can travel between two points in straight line without crossing any ocean or any major water bodies. However, the path crosses through the Dead Sea (which can be considered to be a major water body), and hence does not satisfy the constraints originally set out.
\section{Problem Statement}\label{assumptions}
Thus, in an effort to prove or disprove \texttt{kepleronlyknows}, and to find the longest straight line path on land without crossing a major water body, the authors set out to systematically arrive at the globally optimal solutions to the problem.

The land/water data for the whole earth required to set up the problem was not available readily. In the end, the data set used was the ETOPO1 Global Relief Model \cite{etopo1}, and consider any negative altitude to be water and positive altitude to be land. Although this is not strictly true in general, in the absence of actual data it is the closest approximation to real relief that is freely available.
\subsection{ETOPO1 Global Relief Model}
ETOPO1 is a 1 arc-minute global relief model of Earth's surface that integrates land topography and ocean bathymetry. It was built from numerous global and regional data sets, and is available in ``Ice Surface'' (top of Antarctic and Greenland ice sheets) and ``Bedrock'' (base of the ice sheets) versions. The version suited for this problem was the ``Ice Surface'' one, since a sailable path should not hit an ice layer, and technically ice could be driven upon.
\subsection{Optimisation}
The real crux of the problems lay in solving the optimisation problem globally. After an initial aborted attempt to do a brute-force search across the paths, the authors settled on utilising the branch-and-bound algorithm.
\section{Problem Formulation}
The first step in formulating the problem was to implement a method for enumerating the different straight line paths on the Earth. To do this, it is important to note that all straight line paths lie on a great circle. Each great circle can have as many paths as there are crossings of the great circle with the land-water boundary, but since the interest lies in only the longest paths (either land or water) in each great circle, enumerating all the great circles would be enough. This can be done in a systematic way by observing that any great circle (except the equator itself) intersects the equator at two antipodal points, and at a specific angle. Excepting the trivial case of the equator, the enumeration can thus be chosen by choosing the longitude of the intersection with the equator (the ``origin''), and the angle between the great circle and the equator at that point (the ``heading''), either of which could be chosen between $0$ and $360^\circ$. This counts each great circle four times, however, so the actual choice can be restricted to between $0$ and $180^\circ$ for each. Each point of the great circle can then be pinpointed by the angle of the vector in the plane of the great circle (the reference vector can be arbitrarily set to the origin). The domain of this angle ($\phi$) is between $0$ and $360^\circ$.

Any great circle defined by the origin ($\delta$) and the heading ($\alpha$) can thus be denoted by $\text{GC}\left(\delta,\alpha\right)$.

To match the ETOPO1 data set, which is given as altitude in meters at each latitude and longitude, each point ($\phi$) on each great circle $\text{GC}\left(\delta,\alpha\right)$ needs to be translated into latitude ($\beta$) and longitude ($\lambda$), which can be done using basic spherical trigonometry:
\begin{align}
  \beta&=\sin^{-1}\left(\sin\alpha\sin\phi\right),\\
  \lambda&=\delta+\Tan^{-1}\left(\cos\alpha\sin\phi,\cos\phi\right),
\end{align}
where the $\Tan^{-1}\left(y,x\right)$ is the four-quadrant arctangent function.

Traversing each great circle defined by $\left(\beta,\lambda\right)$, the relief encountered $r_{\left(\delta,\alpha\right)}\left(\phi\right)$, $\phi\in\left(-180^\circ,180^\circ\right]$ is given by
  \begin{equation}
    r_{\left(\delta,\alpha\right)}\left(\phi\right)=\textsc{etopo1}\left(\sin^{-1}\left(\sin\alpha\sin\phi\right),\delta+\Tan^{-1}\left(\cos\alpha\sin\phi,\cos\phi\right)\right)
  \end{equation}

  For example, as shown in Figure \ref{fig:globe}, $\text{GC}\left(163^\circ44',151^\circ22'\right)$ passes through both the Challenger Deep ($11^\circ20'$ N, $142^\circ12'$ E) at an angular distance of $24^\circ12'$ from the origin, and Mount Everest ($27^\circ59'$ N, $86^\circ56'$ E) at an angular distance of $78^\circ22'$ from the origin.

  \begin{figure}[htbp]
    \centering
    \includegraphics[width=0.5\textwidth]{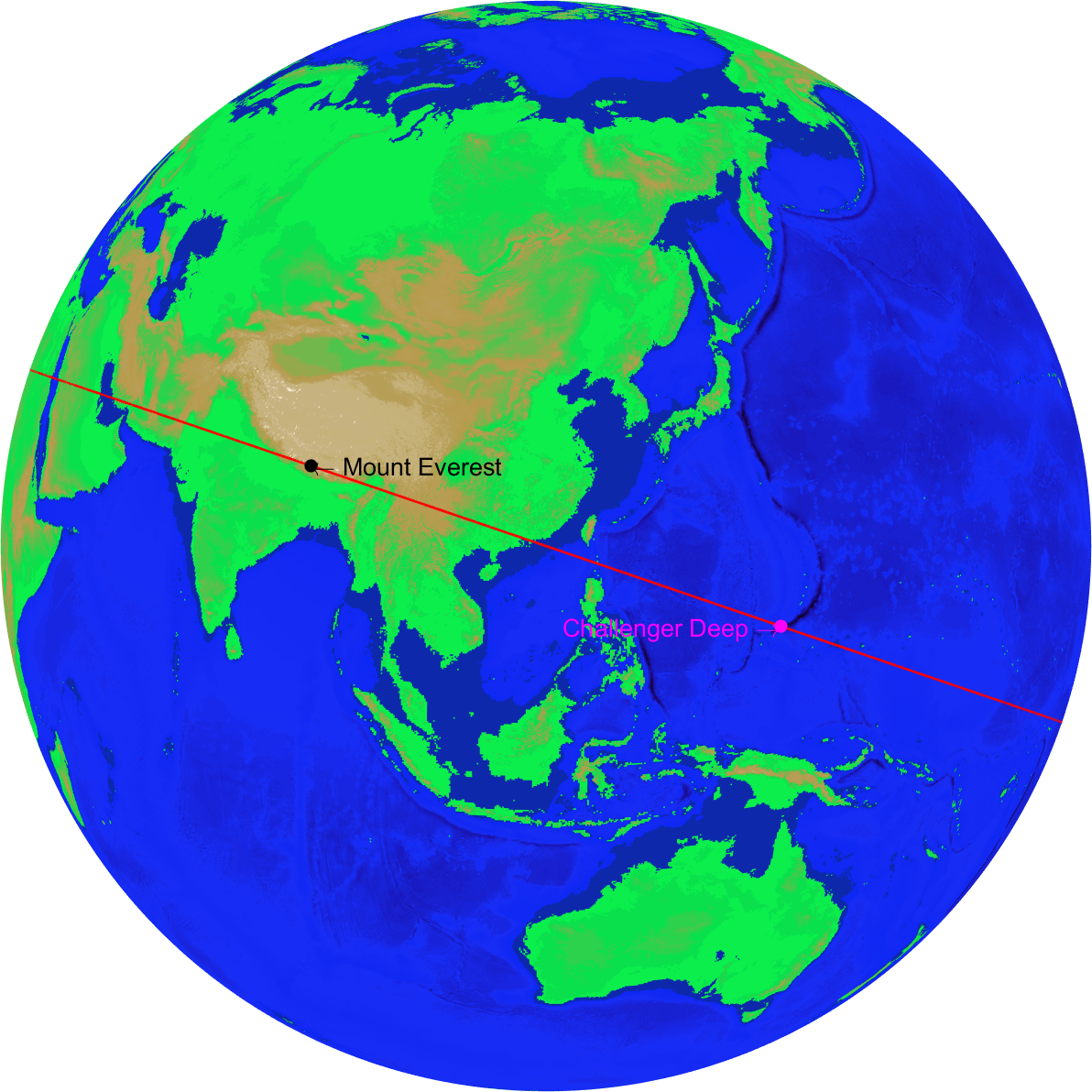}
    \caption{Great circle passing through the lowest and highest points on the Earth.}
    \label{fig:globe}
  \end{figure}

  The relief encountered $r_{\left(163^\circ44',151^\circ22'\right)}\left(\phi\right)$, $\phi\in\left(-180^\circ,180^\circ\right]$ is shown in Figure \ref{fig:relief}.

    \begin{figure}[htbp]
      \centering
      \subfloat[Actual relief]{\includegraphics[width=\textwidth]{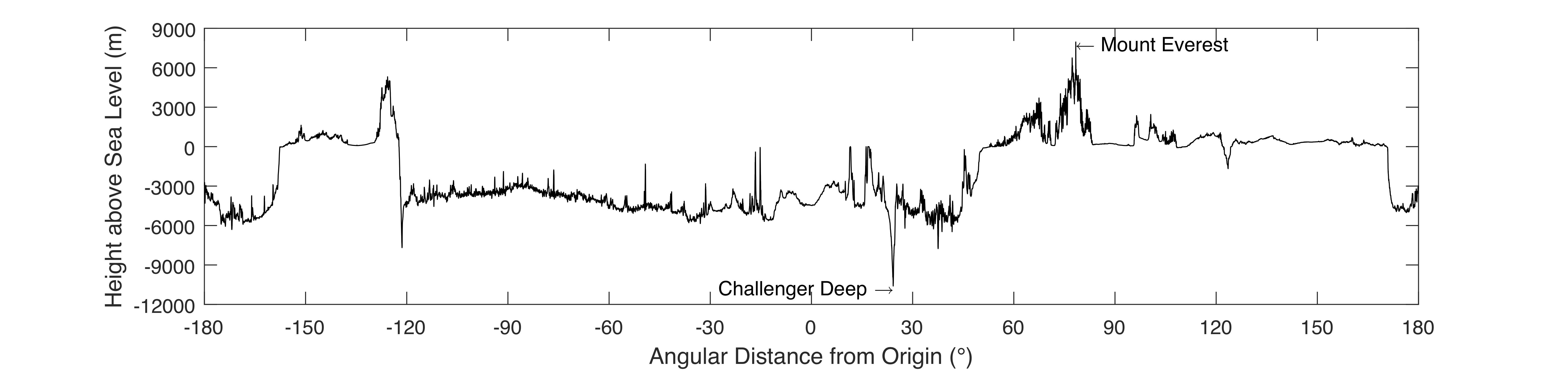}\label{fig:relief}}\\
      \subfloat[Land/Water classification]{\includegraphics[width=\textwidth]{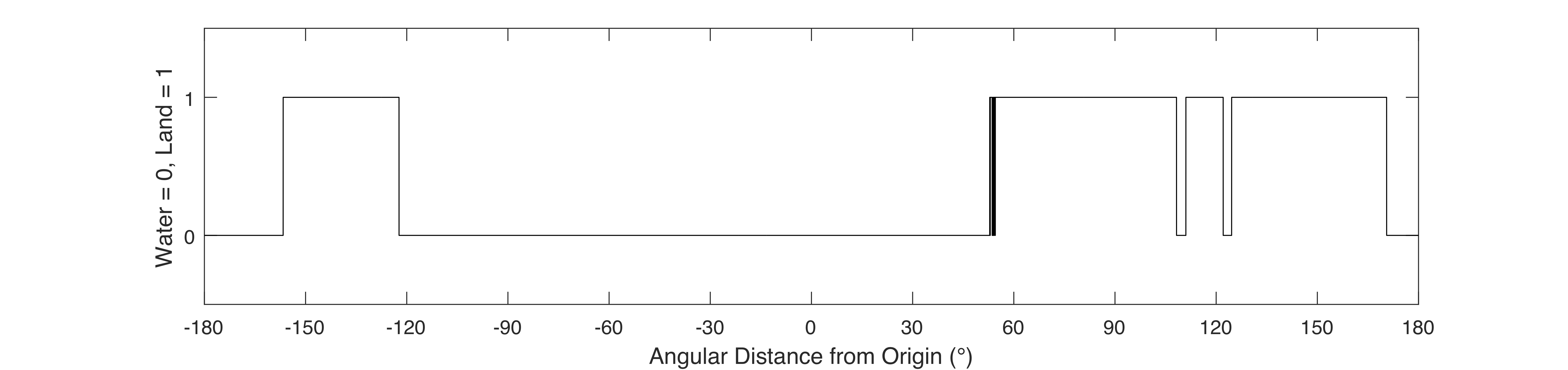}\label{fig:lorw}}
      \caption{Relief along the great circle passing through the lowest and highest points on the Earth.}
    \end{figure}

    Given the relief along each great circle, the assumption of land and water being given by positive and negative heights above sea level can be applied to classify each point of the great circle being either land or water:

    \begin{equation}
      \text{land}_{\left(\delta,\alpha\right)}\left(\phi\right)=\begin{cases}1&\mbox{ if } r_{\left(\delta,\alpha\right)}\left(\phi\right)>0\\0&\mbox{ if } r_{\left(\delta,\alpha\right)}\left(\phi\right)<0\\\text{indeterminate}&\mbox{ if } r_{\left(\delta,\alpha\right)}\left(\phi\right)=0\end{cases}.
    \end{equation}

    The indeterminacy is resolved conservatively by assuming $r_{\left(\delta,\alpha\right)}\left(\phi\right)=0$ to be land when searching for the longest water path and water when searching for the longest land path. For the example great circle, $\text{land}_{\left(163^\circ44',151^\circ22'\right)}\left(\phi\right)$, $\phi\in\left(-180^\circ,180^\circ\right]$ is shown in Figure \ref{fig:lorw}.

      Once the classification is obtained, it is a simple matter to determine the longest land and water paths in the great circle, which form the objective functions that need to be maximised.

      \section{Branch-And-Bound}
      Branch-and-bound is a method to solve combinatorial optimisation problems, especially when it is possible to get bounds on the objective function for a subset in the space of solutions. 

      For the problem addressed in this paper, we continue to denote each great circle by the pair of angles that uniquely characterise it as $\text{GC}\left(\delta,\alpha\right)$. A \emph{set} of solutions refers to a set of great circle trajectories that intersect the equator between latitudes $\delta_{\min}$ and $\delta_{\max}$ and with heading in the range $\alpha_{\min}$ to $\alpha_{\max}$.  We denote the set as $\left[\delta_{\min},\delta_{\max}\right]\times\left[\alpha_{\min},\alpha_{\max}\right]$.
      Each great circle $\text{GC}\left(\delta,\alpha\right)$ has an associated length $l\left(\delta,\alpha\right)$ that is defined as the maximum distance that can be traversed along the great circle without hitting land.  

      \subsection{Relaxation Techniques}
      The next task is to obtain an upper bound on the length of longest continuous section of sea $l\left(\delta,\alpha\right)$ along all great circles such that $\delta\in\left[\delta_{\min},\delta_{\max}\right],\alpha\in\left[\alpha_{\min},\alpha_{\max}\right]$.

      \begin{theorem}
        Considering $\text{GC}\left(\frac{\delta_{\min}+\delta_{\max}}{2},\frac{\alpha_{\min}+\alpha_{\max}}{2}\right)$ as representative of the set of great circles uch that $\delta\in\left[\delta_{\min},\delta_{\max}\right],\alpha\in\left[\alpha_{\min},\alpha_{\max}\right]$, the maximum angular separation between the representative great circle and any other great circle in the set is always bounded by $\left|\frac{\delta_{\max}-\delta_{\min}}{2}\right|+\left|\frac{\alpha_{\max}-\alpha_{\max}}{2}\right|$.
        \label{thm:upperbound}
      \end{theorem}
      \begin{proof}
        Let $\theta=\frac{\alpha_{\max}+\alpha_{\min}}{2}$, $\Delta\theta=\alpha_{\max}-\alpha_{\min}$, and $\Delta\phi=\delta_{\max}-\delta_{\min}$. Using spherical trigonometry, the maximum angular separation between them is $\cos^{-1}\left(\cos\left(\Delta\theta\right)\cdot\cos^2\left(\frac{\Delta\phi}{2}\right)+\cos\left(2\theta\right)\cdot\sin^2\left(\frac{\Delta\phi}{2}\right)\right)$.

        This is non-trivial to use, however, the maximum deviation occurs when $\theta=90^\circ$, and has a theoretical upper limit of $\sqrt{\Delta\theta^2+\Delta\phi^2}$, as $\Delta\theta,\Delta\phi\to0^\circ$. Even for $\Delta\theta,\Delta\phi=45^\circ$, the limit is accurate to $2.62\%$. Since we know that $\left\|\Delta\theta\right\|+\left\|\Delta\phi\right\|\geq\sqrt{\Delta\theta^2+\Delta\phi^2}$, the angular separation is always bounded by this value, which is, $\left|\frac{\delta_{\max}-\delta_{\min}}{2}\right|+\left|\frac{\alpha_{\max}-\alpha_{\max}}{2}\right|$.
      \end{proof}
      
      Consider the great circle $\text{GC}\left(\frac{\delta_{\min}+\delta_{\max}}{2},\frac{\alpha_{\min}+\alpha_{\max}}{2}\right)$ as representative of the set of great circles. Using Theorem~\ref{thm:upperbound}, the maximum separation between the representative great circle and any other great circle in the set is always bounded by $R\cdot\left(\left|\frac{\delta_{\max}-\delta_{\min}}{2}\right|+\left|\frac{\alpha_{\max}-\alpha_{\max}}{2}\right|\right)$, where $R$ is the radius of the earth.

      The novelty in our approach lies in a method for evaluation the upper bound on the longest continuous section of sea. Let there be a great circle $\text{GC}\left(\delta_o,\alpha_o\right)$ where $\delta_o\in\left[\delta_{\min},\delta_{\max}\right]$ and $\alpha_o\in\left[\alpha_{\min},\alpha_{\max}\right]$ which has the longest continuous section of sea.

      \begin{equation}
        l\left(\delta_o,\alpha_o\right)=\max_{\substack{\delta\in\left[\delta_{\min},\delta_{\max}\right]\\\alpha\in\left[\alpha_{\min},\alpha_{\max}\right]}}l\left(\delta,\alpha\right). 
      \end{equation}

      That implies,

      \begin{equation}
        l\left(\delta,\alpha\right)\leq\max_{\substack{\delta\in\left[\delta_{\min},\delta_{\max}\right]\\\alpha\in\left[\alpha_{\min},\alpha_{\max}\right]}}l\left(\delta,\alpha\right),~\forall \delta \in \left[\delta_{\min},\delta_{\max}\right], \alpha \in \left[\alpha_{\min}, \alpha_{\max}\right].
      \end{equation}

      Further, consider the globe where the sea has eroded into the land by a distance $\epsilon$. In this new globe ($\epsilon$-eroded globe), the land masses have shrunk by a distance of epsilon from all sides. All pieces of land, that are within a distance of $\epsilon$ from the sea, are covered by sea on the $\epsilon$-globe. We define $l_{\epsilon}\left(\delta,\alpha\right)$ as the maximum distance that can be traversed along the great circle without hitting land on the $\epsilon$-eroded globe. Clearly,

      \begin{equation}
        l_{\epsilon}\left(\delta,\alpha\right)\geq\max_{\substack{\delta\in\left[\delta_{\min},\delta_{\max}\right]\\\alpha\in\left[\alpha_{\min},\alpha_{\max}\right]}}l\left(\delta,\alpha\right),~\forall \delta \in \left[\delta_{\min},\delta_{\max}\right], \alpha \in \left[\alpha_{\min}, \alpha_{\max}\right].
      \end{equation}

      Let's consider the case where $\epsilon\geq R\left(\left|\frac{\delta_{\max}-\delta_{\min}}{2}\right|+\left|\frac{\alpha_{\max}-\alpha_{\max}}{2}\right|\right)$. That is, the extent of erosion is larger that the maximum distance between any great circle in the set and the representative great circle. Since, the distance between the great circle $\left(\delta_o,\alpha_o\right)$ and the representative great circle is less that $\epsilon$, the longest section (over sea) of the great circle $\left(\delta_o,\alpha_o\right)$ would erode pieces of land that lie on the representative great circle. Thus, the longest section (over sea) of the representative great circle on the $\epsilon$-eroded globe would be at least as large as $l\left(\delta_o,\alpha_o\right)$.

      \begin{equation}
        l_{\epsilon}\left(\frac{\delta_{\min}+\delta_{\max}}{2},\frac{\alpha_{\min}+\alpha_{\max}}{2}\right)\geq l\left(\delta,\alpha\right).
      \end{equation}

      Hence the upper bound of $l\left(\delta,\alpha\right),\forall\delta \in \left[\delta_{\min},\delta_{\max}\right],\alpha\in\left[\alpha_{\min},\alpha_{\max}\right]$ can be evaluated as the length of the longest stretch of the representative great circle on the $\epsilon$-eroded globe such that $\epsilon\geq R\left(|\frac{\delta_{\max}-\delta_{\min}}{2}|+ |\frac{\alpha_{\max}-\alpha_{\max}}{2}|\right)$.

      The erosion should technically be carried out in three dimensions (assuming that the angular land-water boundaries continue till the centre). However, this would be computationally intensive, to store and calculate at the resolution required. Since the ETOPO1 database is equivalent to the equirectangular projection, in the two-dimensional context, the eroding element would have to be continuously deformed as per the Tissot's indicatrix for equirectangular projection, as shown in Figure~\ref{fig:tissot}. Continuously deforming the eroding element takes the problem to using non-standard algorithms for erosion.

      \begin{figure}[htbp]
        \centering
        \includegraphics[width=\textwidth]{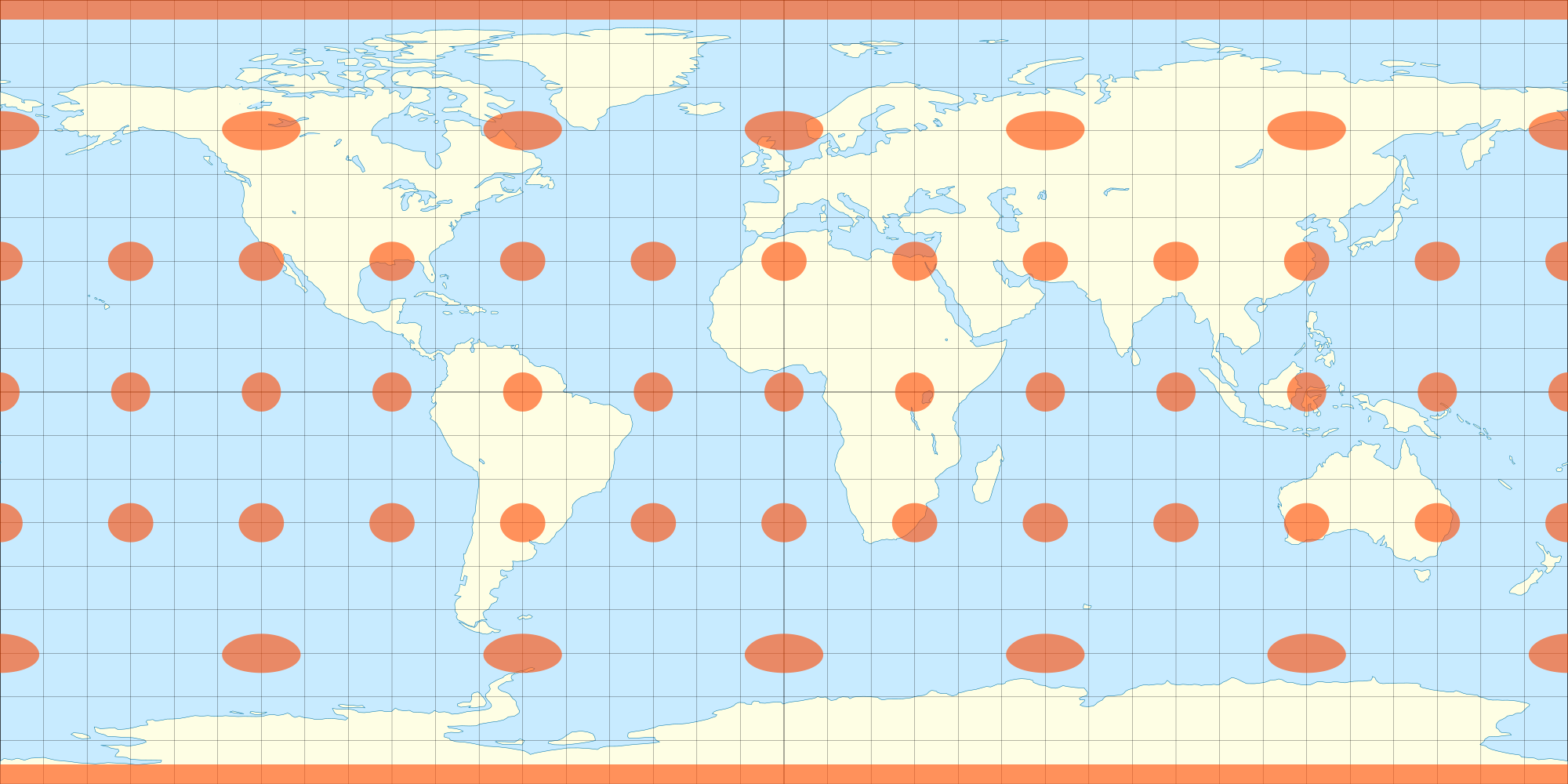}
        \caption{The equirectangular projection with Tissot's indicatrix of deformation, from \cite{tissot}.}
        \label{fig:tissot}
      \end{figure}

      Instead, for speed and complexity, we choose to use a single over-estimated erosion radius for the whole map. Over-estimating the erosion radius reduces the relaxation, and hence might potentially affect the computation required, but does not affect the results. The factor for over-estimation, however, needs to be chosen with care. It is certain, due to the existence of Antarctica (for water), and the absence of land near the North Pole, that the great circle with maximum distance won't extend beyond $78^\circ28'$ North or South. Thus the erosion only needs to be over-estimated till that latitude, giving us a reasonable factor of 5.

      \subsection{Branch-And-Bound Algorithm}

      Algorithm \ref{bbalg} depicts the standard branch-and-bound algorithm that has been developed to address the problem of finding the longest straight line path over sea. The successive branching of the set of trajectories into four subsets (by dividing the origin and heading into two subsets each) based on their upper bounds enables the algorithm to quickly converge at the solution.

      \begin{algorithm}
        \caption{Algorithm to obtain the great circle with the longest section over sea}\label{bbalg}
        \begin{algorithmic}[1]
          \State Initialise \texttt{Nodes} with one element: $\left(\delta_{\min}=0,\delta_{\max}=\pi,\alpha_{\min}=0,\alpha_{\max}=\pi,\epsilon=\pi R,B=2\pi R\right)$.
          \State Remove element $\left(\delta_{\min}^c,\delta_{\max}^c,\alpha_{\min}^c,\alpha_{\max}^c,\epsilon^c,B^c\right)$ with largest B from \texttt{Nodes}.
          \While {$\epsilon^c\geq\epsilon_{\text{threshold}}$}
          \State Add $\left(\delta_{\min}^c,\frac{\delta_{\max}^c+\delta_{\min}^c}{2},\alpha_{\min}^c,\frac{\alpha_{\max}^c+\alpha_{\min}^c}{2},\epsilon=\frac{\epsilon}{2},B=l_{\epsilon}\left(\frac{3\delta_{\min}^c+\delta_{\max}^c}{4},\frac{3\alpha_{\min}^c+\alpha_{\max}^c}{4}\right)\right)$ to \texttt{Nodes}.
          \State Add $\left(\frac{\delta_{\max}^c+\delta_{\min}^c}{2},\delta_{\max}^c,\alpha_{\min}^c,\frac{\alpha_{\max}^c+\alpha_{\min}^c}{2},\epsilon=\frac{\epsilon}{2},B=l_{\epsilon}\left(\frac{\delta_{\min}^c+3\delta_{\max}^c}{4},\frac{3\alpha_{\min}^c+\alpha_{\max}^c}{4}\right)\right)$ to \texttt{Nodes}.
          \State Add $\left(\delta_{\min}^c,\frac{\delta_{\max}^c+\delta_{\min}^c}{2},\frac{\alpha_{\max}^c+\alpha_{\min}^c}{2},\alpha_{\max}^c,\epsilon=\frac{\epsilon}{2},B=l_{\epsilon}\left(\frac{3\delta_{\min}^c+\delta_{\max}^c}{4},\frac{\alpha_{\min}^c+3\alpha_{\max}^c}{4}\right)\right)$ to \texttt{Nodes}.
          \State Add $\left(\frac{\delta_{\max}^c+\delta_{\min}^c}{2},\delta_{\max}^c,\frac{\alpha_{\max}^c+\alpha_{\min}^c}{2},\alpha_{\max}^c,\epsilon=\frac{\epsilon}{2},B=l_{\epsilon}\left(\frac{\delta_{\min}^c+3\delta_{\max}^c}{4},\frac{\alpha_{\min}^c+3\alpha_{\max}^c}{4}\right)\right)$ to \texttt{Nodes}.
          \State Remove element $\left(\delta_{\min}^c,\delta_{\max}^c,\alpha_{\min}^c,\alpha_{\max}^c,\epsilon^c,B^c\right)$ with largest B from \texttt{Nodes}.
          \EndWhile    
          \State Return $\left(\delta_{\min}^c,\delta_{\max}^c,\alpha_{\min}^c,\alpha_{\max}^c,\epsilon^c,B^c\right)$.
        \end{algorithmic}
      \end{algorithm}

      \section{Results}
      \subsection{Erosion}
      \begin{figure}[htbp]
        \centering
        \subfloat[Water Mass Erosion]{\includegraphics[width=\textwidth]{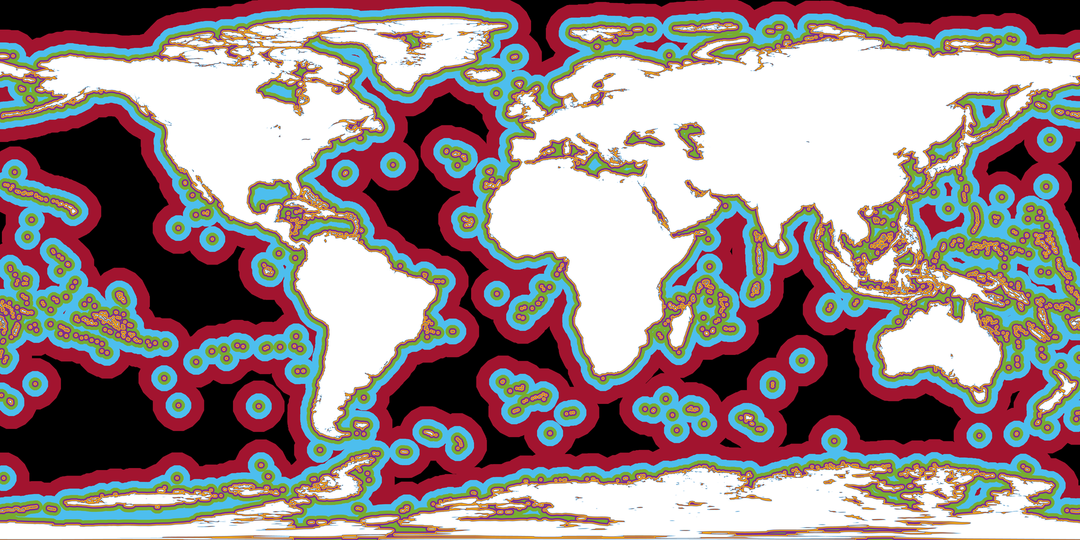}\label{fig:erosion1}}\\
        \subfloat[Land Mass Erosion]{\includegraphics[width=\textwidth]{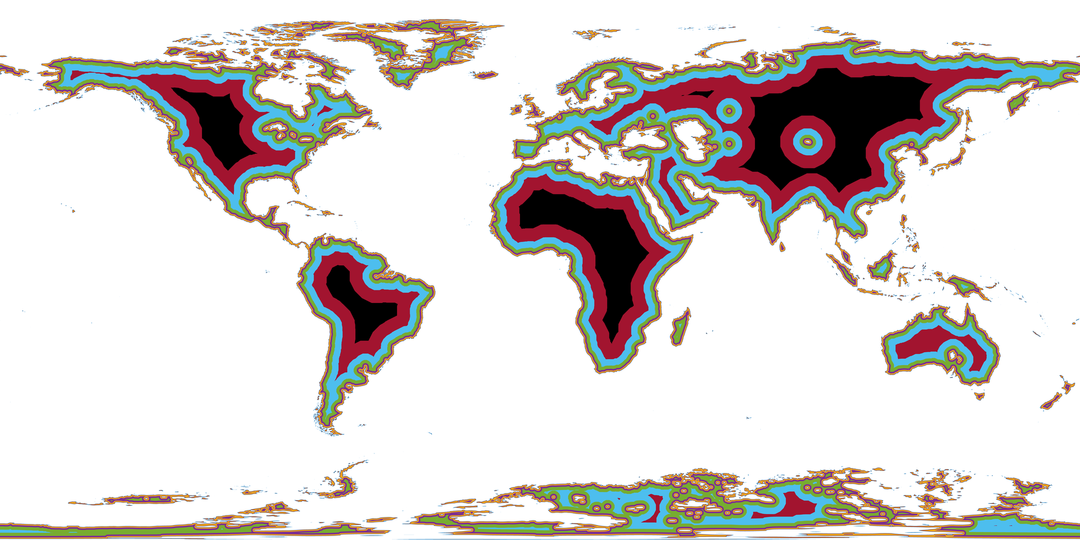}\label{fig:erosion2}}
        \caption{Erosion steps 3--7 for water (\ref{fig:erosion1}) and land (\ref{fig:erosion2}) masses respectively.}
      \end{figure}

      The erosion steps 3--7 for both water and land can be seen in Figures~\ref{fig:erosion1} and~\ref{fig:erosion2} respectively. The nuances of the erosion are better seen by zooming in on the British Isles, as has been done in Figures~\ref{fig:erosion1_ireland} and~\ref{fig:erosion2_ireland}.

      \begin{figure}[htbp]
        \centering
        \subfloat[Water Mass Erosion]{\includegraphics[width=\textwidth]{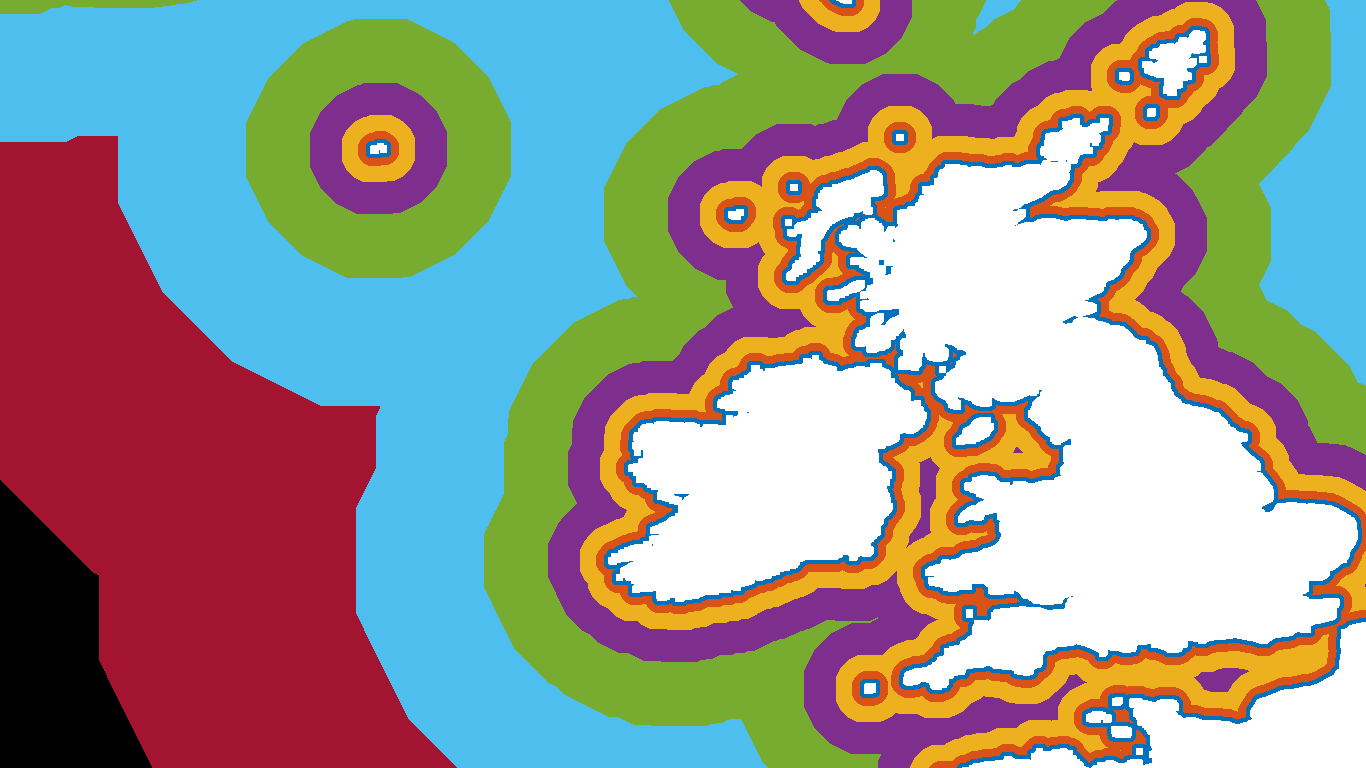}\label{fig:erosion1_ireland}}\\
        \subfloat[Land Mass Erosion]{\includegraphics[width=\textwidth]{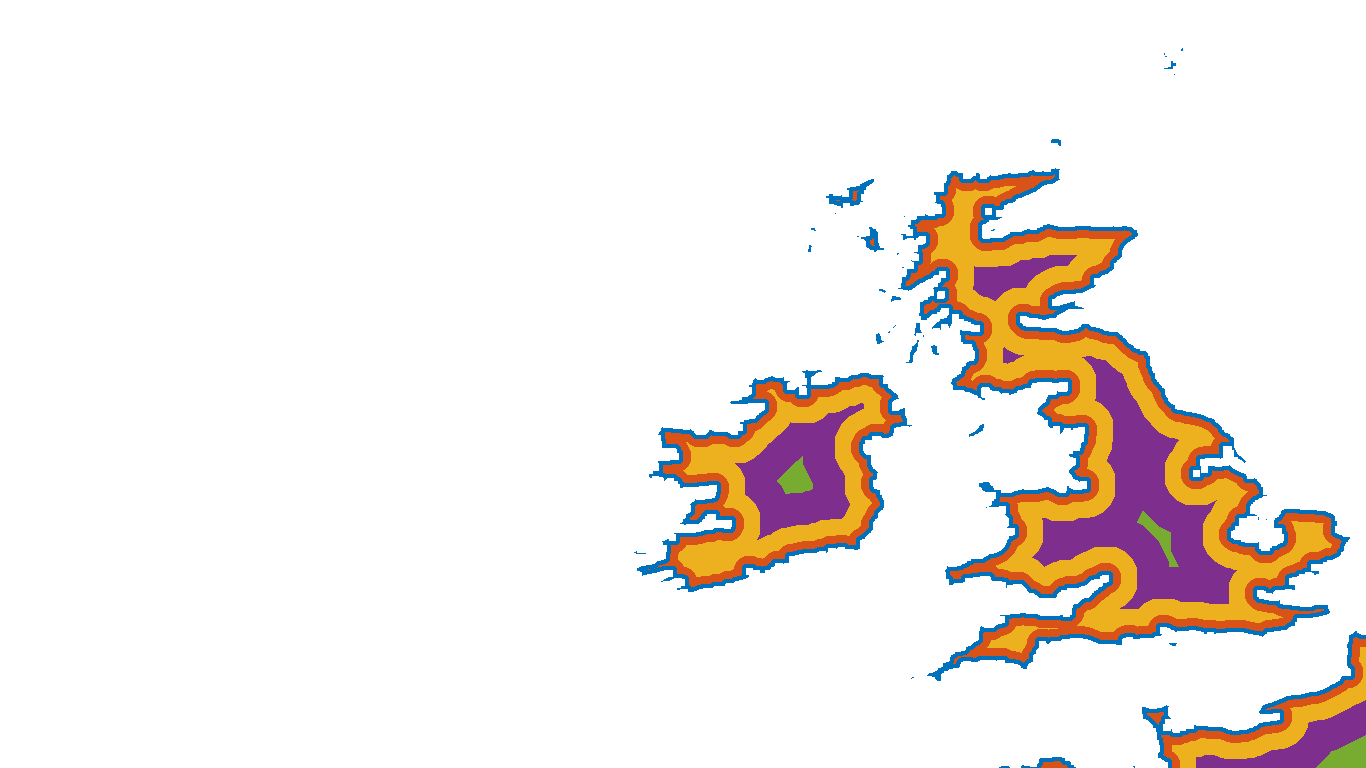}\label{fig:erosion2_ireland}}
        \caption{Erosion steps 3--7 for water (\ref{fig:erosion1_ireland}) and land (\ref{fig:erosion2_ireland}) masses respectively, around the British Isles.}
      \end{figure}

      \subsection{Longest Paths}
      The algorithm returned the longest path in about 10 minutes of computation for water path, and 45 minutes of computation for land path on a standard laptop.
      \subsubsection{Longest Sailable Straight Line Path on Earth}
      The algorithm returned the path shown in Figure~\ref{fig:waterpathmap}. Although it does not look like a straight line on the map, the algorithm using great circles ensures that it is.
      \begin{figure}[htbp]
        \centering
        \includegraphics[width=\textwidth]{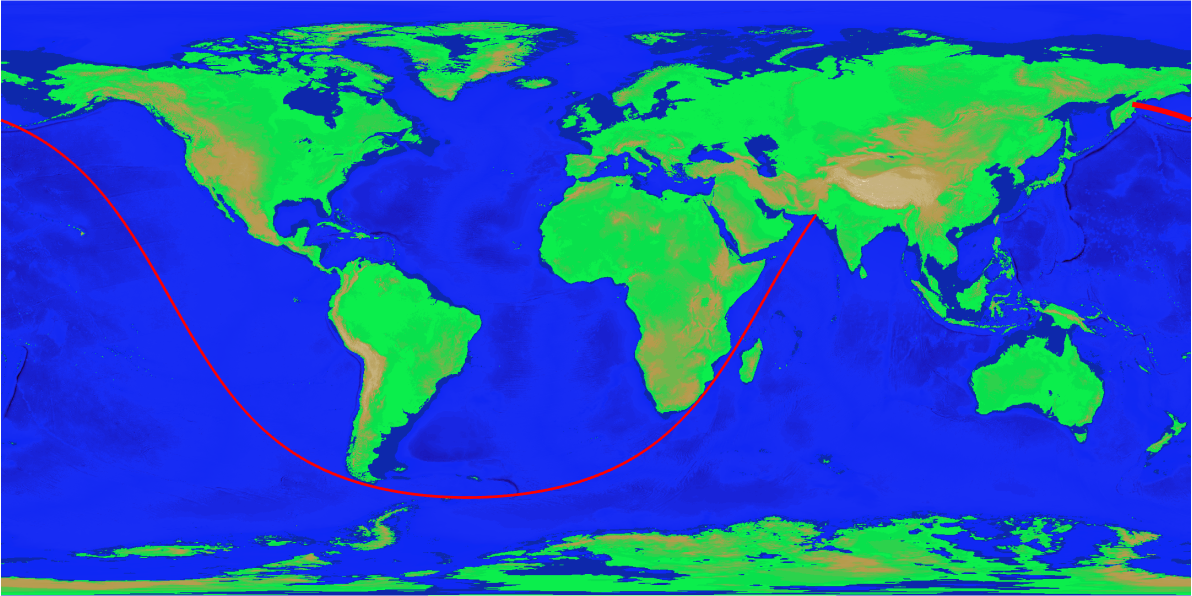}
        \caption{Longest Sailable Straight Line Path on Earth.}
        \label{fig:waterpathmap}
      \end{figure}

      The fact that it is indeed a straight line can be seen from Figures~\ref{fig:africa},~\ref{fig:southamerica}, and~\ref{fig:russia}, which are shown on a globe with a perspective from above the path. Figure~\ref{fig:africa} shows the path originating in Sonmiani, Las Bela, Balochistan, Pakistan ($25^\circ17'$ N, $66^\circ40'$ E), threading the needle between Africa and Madagascar, between Antarctica and Tiera del Fuego in South America (Figure~\ref{fig:southamerica}), and ending in Karaginsky District, Kamchatka Krai, Russia ($58^\circ37'$ N, $162^\circ14'$ E) (Figure~\ref{fig:russia}).

      The path covers an astounding total angular distance of $288^\circ35'$, for a distance of 32 090 kilometres. This path is visually the same one as found by \texttt{kepleronlyknows}, thus proving his assertion.
      
      \begin{figure}[htbp]
        \centering
        \subfloat[Start of Path]{\includegraphics[width=0.3\textwidth]{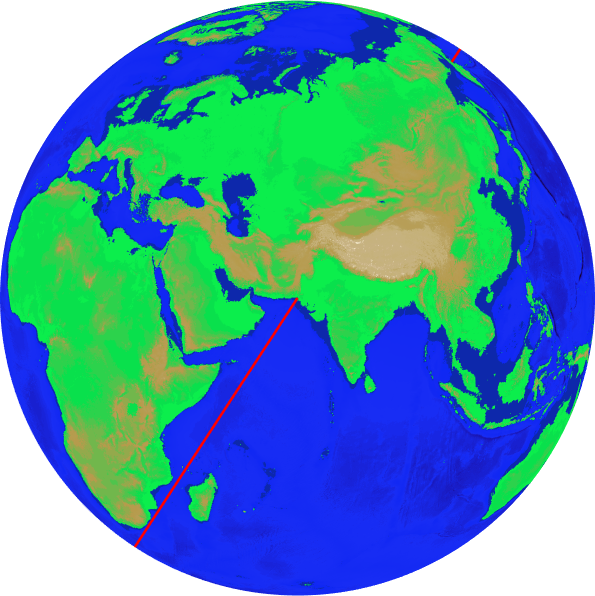}\label{fig:africa}}
        \subfloat[Threading the Needle]{\includegraphics[width=0.3\textwidth]{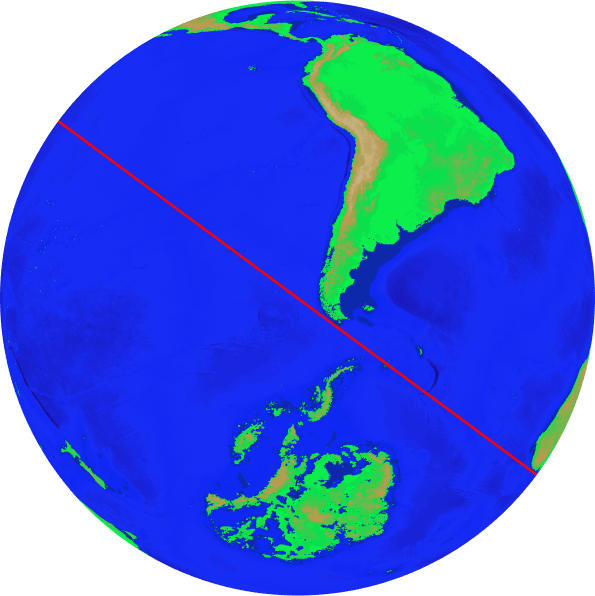}\label{fig:southamerica}}
        \subfloat[End of Path]{\includegraphics[width=0.3\textwidth]{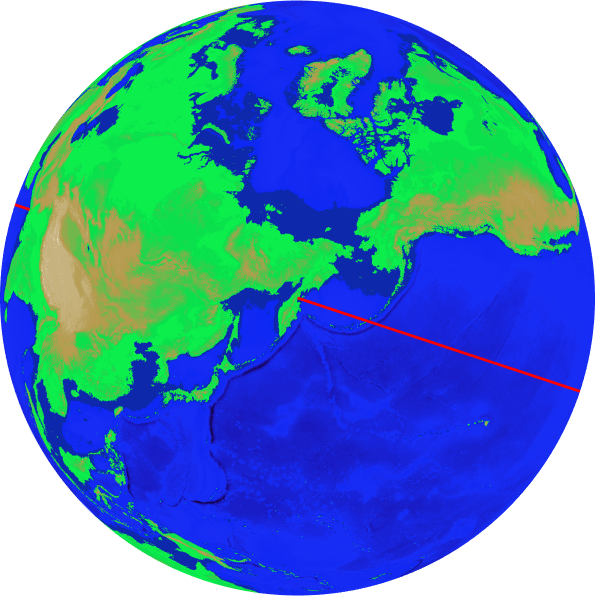}\label{fig:russia}}
        \caption{Longest Sailable Straight Line Path on Earth.}
      \end{figure}

      \subsubsection{Longest Drivable Straight Line Path on Earth}
      The algorithm returned the path shown in Figure~\ref{fig:landpathmap}. Again, although it does not look like a straight line on the map, the algorithm using great circles ensures that it is.
      \begin{figure}[htbp]
        \centering
        \includegraphics[width=\textwidth]{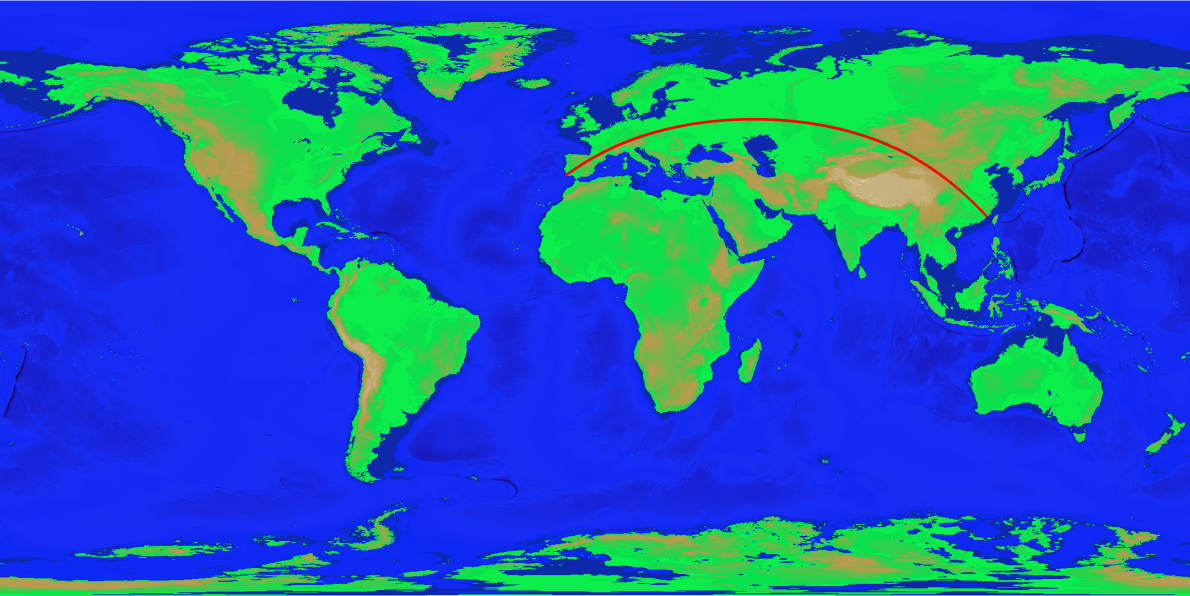}
        \caption{Longest Drivable Straight Line Path on Earth.}
        \label{fig:landpathmap}
      \end{figure}

      The fact that it is indeed a straight line can be seen from Figure~\ref{fig:path2}, which is shown on a globe with a perspective from above the path. Figure~\ref{fig:landpathmap} shows the path originating near Jinjiang, Quanzhou, Fujian, China ($24^\circ33'$ N, $118^\circ38'$ E), weaving through China and Mongolia for a bit, passing though Kazakhstan and Russia to further weave through Belarus and Ukraine, and passing through Poland, Czech Republic, Germany, Austria, Liechtenstein, Switzerland, France, and Spain, to end near Sagres, Portugal ($37^\circ2'$ N $8^\circ55'$ W), traversing a total of 15 countries. 

      Although not as long as the longest sailable path, the longest drivable path covers a still-respectable total angular distance of $101^\circ6'$, for a distance of 11 241 kilometres.

      \begin{figure}[htbp]
        \centering
        \includegraphics[width=0.3\textwidth]{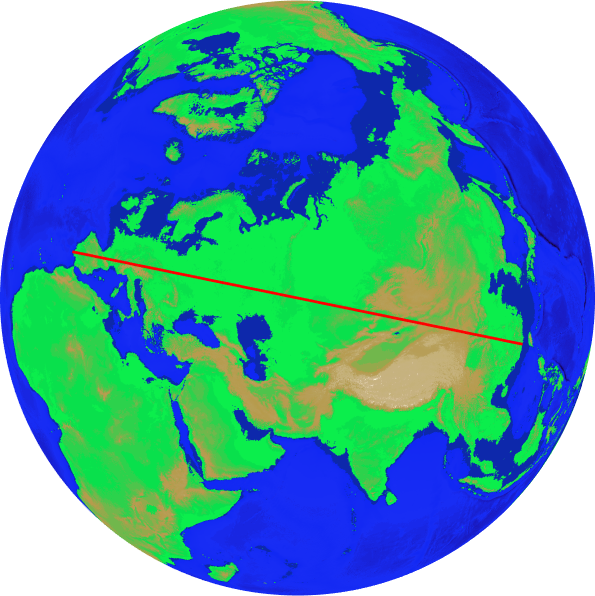}
        \caption{Longest Drivable Straight Line Path on Earth.}
        \label{fig:path2}
      \end{figure}

      \section{Limitations and Future Work}
      As mentioned in Section~\ref{assumptions}, our approach assumes the partition between land and water to be determined by the height with respect to the mean sea level. This unfortunately discounts, among other features, highland rivers, and low-lying plains. We also cannot consider bridges, although the probability of a bridge lying at the precise location and direction of any given path is negligible.

      Secondly, the ETOPO1 database has a resolution of 1-arcminute, about 1.8 kilometres at the equator. The optimisation problem, therefore, has by necessity used this as the lower-limit of accuracy for analysis.

      Both of these factors are the result of the input dataset.

      Thirdly, an implicit assumption is that the Earth is perfectly spherical. This is untrue because of the ellipsoidal shape of the earth, as well as due to geoid undulations.

      Future work would involve running the analysis on a higher resolution dataset, or on one that provides the shape of all the land and water bodies in higher detail. It would also require us to stop using spherical trigonometry, or at least, consider deviations due to the non-spherical nature of the geoid, using local curvatures. Consideration all of these issues, and more, does seem to be an overkill for what is essentially a recreational problem.
      
      \section{Conclusion}

      We proposed an innovative approach for relaxation of an optimisation problem for utilising the branch-and-bound algorithm. On the way, we managed to prove that \texttt{kepleronlyknows} was right about the longest sailable straight line path on the Earth, and we found the analogous longest drivable straight line path.

      The problem was approached as a purely mathematical exercise. The authors do not recommend sailing or driving along the found paths.

\end{document}